\documentclass[twoside,12pt, leqno]{amsart}
\usepackage{amsmath,amscd,amsthm,amssymb,amsxtra,latexsym,epsfig,epic,graphics}
\usepackage[matrix,arrow,curve]{xy}
\usepackage{graphicx}
\usepackage{diagrams}
\usepackage{MnSymbol}
\usepackage{tikz}  %TikZ
\usepackage{color} 
\usetikzlibrary{arrows,calc} 
\usetikzlibrary{decorations.pathmorphing}
\usetikzlibrary{decorations.markings} 
\usetikzlibrary{decorations.pathreplacing} 
\usetikzlibrary{plothandlers}

\def\antiddot{\mathinner{\mkern1mu\raise1pt\vbox{\kern7pt\hbox{.}}\mkern2mu
        \raise4pt\hbox{.}\mkern2mu\raise7pt\hbox{.}\mkern1mu}}

\newcommand{\EE}{{\mathbb E}}
\newcommand{\FF}{{\mathbb F}}
\newcommand{\GG}{{\mathbb G}}

\newcommand{\KK}{{\mathbb K}}
\newcommand{\MM}{{\mathbb M}}

\newcommand{\TT}{{\mathbb T}}

\newcommand{\LL}{{\mathbb L}}
\newcommand{\HH}{{\rm{H}}}

\newcommand{\Ext}{{\rm{Ext}}}

\newcommand{\punkt}{\hspace{-.3ex}\raise.15ex\hbox to1ex{\Huge.}}

\def \fix#1 {{\hfill\break \bf (( #1 ))\hfill\break}}

\DeclareMathOperator{\Hom}{Hom}

\DeclareMathOperator{\depth}{depth}

\DeclareMathOperator{\Tor}{Tor}

\newtheorem{theorem}{Theorem}[section]
\newtheorem{lemma}[theorem]{Lemma}
\newtheorem{proposition}[theorem]{Proposition}
\newtheorem{corollary}[theorem]{Corollary}

\theoremstyle{definition}

\newtheorem{example}[theorem]{Example}

\def\FF{{\mathbb F}}

\def\fix#1{{\bf ***Fix:} #1 {\bf ***}}

\def\mm{{\frak m}}

\newarrow{Iso} -----

\def\D{{\mathcal D}}
\def\S{{\mathcal S}}

\def\lbracket{{[\kern-1.5pt[}}
\def\rbracket{{]\kern-1.5pt]}}

\def\seq#1#2{{#1_{1},\dots,#1_{#2}}}

\def\cn{{n}}

\makeatletter
\def\Ddots{\mathinner{\mkern1mu\raise\p@
\vbox{\kern7\p@\hbox{.}}\mkern2mu
\raise4\p@\hbox{.}\mkern2mu\raise7\p@\hbox{.}\mkern1mu}}
\makeatother

\usepackage{times}
\newdimen\x \x=12pt

\usepackage{color}

\usepackage[breaklinks,bookmarksopen,bookmarksnumbered,urlcolor=blue]{hyperref}
\hypersetup{colorlinks=true,backref=true,citecolor=blue}

\author[David Eisenbud]{David Eisenbud}
\address{Department of Mathematics, University of California at Berkeley and the Mathematical
Sciences Research Institute, Berkeley, CA 94720, USA}
\email{de@msri.org}

\author[Frank-Olaf Schreyer]{Frank-Olaf Schreyer}
\address{Fachbereich Mathematik, Universit\"at des Saarlandes, Campus E2 4, D-66123 Saar\-br\"ucken, Germany}
\email{schreyer@math.uni-sb.de}

\title{Tate Resolutions and Maximal Cohen-Macaulay Approximations}
\begin{document}

\begin{abstract}
We study the Tate resolutions and the maximal Cohen-Macaulay approximations of Cohen-Macaulay modules over Gorenstein rings. One consequence is an extension of a well-known result about linkage of complete intersections.
\end{abstract}

\maketitle

\section*{Introduction} 
\subsection*{Tate Resolutions} Let $R$ be a Gorenstein local ring, and let $M$ be a finitely generated $R$-module. 
 A \emph{Tate resolution} $\TT$ of $M$ is by definition a free, doubly infinite complex that coincides with a free resolution of $M$ in sufficiently high homological degree. If $\TT$ is minimal, in the sense that the differential of
 $R/\mm_R \otimes \TT$ is 0, then we speak of a minimal Tate resolution. Such minimal Tate resolutions always exist, and are unique up to (typically non-unique) isomorphism. 
  
In this note we make the construction of Tate resolutions more explicit. For example, in the special case when:
\medbreak
\noindent *)\quad $S$ is a Gorenstein ring; $I = (\seq gc)\subset J = (\seq f\cn)$
 are complete intersection ideals; $R = S/I$ and $M = S/J$,
\medskip
 
 \noindent we can specify the Tate resolution of $M$ over $R$ completely (Theorem~\ref{CM approx for 2 reg seqs}), thus extending a Theorem of Tate~\cite{Tate} (see Example~\ref{tate's resolution}). 
 
 More generally, but
somewhat less explicitly, If $R$ is a Gorenstein factor ring of a regular local ring $S$ and $M$ is a Cohen-Macaulay $R$-module (of any codimension), we specify the form of a Tate resolution of $M$ over $R$ (Theorem~\ref{sigma} and, for the case where $R$ is a complete intersection, Proposition~\ref{compositions}). In the case where 
$R = S/(g)$ is a hypersurface, our construction reduces to the well-known construction of a matrix factorization of $g$ made from the even and odd parts of the $S$-free resolution of $M$; but in case $R$ is of higher codimension the construction seems to be new.

Tate resolutions originally appeared in John Tate's study of group cohomology. Generalizing the case of an elementary abelian $p$-group in characteristic $p$, Tate gave an explicit construction of the minimal
 resolution of the residue field of a complete intersection~\cite{Tate}. In the 0-dimensional case,
this resolution, combined with its dual, is the Tate resolution (see Example~\ref{tate's resolution}).

\subsection*{Maximal Cohen-Macaulay Approximations}
Auslander and Buchweitz \cite{AB} defined the \emph{maximal Cohen-Macaulay (MCM) approximation} of $M$ to be the unique minimal surjection from an MCM $R$-module $N$ such that
the kernel of $N\to M$ has finite projective dimension (see \cite{EP} for a recent summary and application of the theory). The module $N$ may always be decomposed as the direct sum of a free module and an MCM $R$-module $M'$ with no free summand. The module $M'$, with its induced map to $M$, is called the 
 \emph{essential MCM approximation} of $M$. It may be constructed by taking the minimal $k$-th syzygy of the dual into $R$ of the minimal $k$-th syzygy of $M$, where $k$ is any number $k \geq \max(2, \depth R-\depth M).$ 
 
The essential MCM approximation $M'$ of $M$ is immediately seen to be the cokernel of the first differential in the minimal Tate resolution of $M$.
The original motivation for our work was to give, in the situation of *) above, an explicit construction of the essential MCM approximation of $S/J$ as a module over $S/I$. This is done in Corollary~\ref{explicit MCM}.

 \begin{example}\label{tate's resolution} Here is a case treated by Tate, in a presentation adapted to this paper:
Suppose that $R = S/(\seq gc)$ where $S$ is a regular local ring with maximal ideal $\mm = (x_{1}\dots,x_{c})$ and that $\seq gc$ is a maximal regular sequence in $\mm^{2}$. Suppose further that $M$ is the residue field of $R$. Tate's paper \cite{Tate} provides an explicit minimal free resolution that may be written as the total complex of a double complex
beginning
\begin{small}
$$
\begin{diagram}[small]
\FF:\quad  R&\lTo^{
\begin{pmatrix}
 x_{1}&\dots&x_{c}
\end{pmatrix}
} &R^{c}&\lTo&\bigwedge^{2}R^{c}&\cdots\\
&&\uTo&&\uTo&\cdots\\
&&R^{c}\otimes R&\lTo&R^{c}\otimes R^{c}&\cdots\\
&&&&\uTo&\cdots
\end{diagram}
$$
\end{small}
Here we have written $R^{c} \otimes R$ instead of $R^{c}$  to emphasize that the second row is the
tensor product of $R^{c}$ with the first row. This is explained in detail in a more general case in Section~\ref{two regs}. 

In this case, $R$ is 0-dimensional and Gorenstein. Thus $R$ is an injective $R$-module, so the dual $\FF^{*} =\Hom_{R}(\FF,R)$ is exact except at $F_{0}^{*}$. Furthermore,
$$
H^{0}(\FF^{*}) = \ker \bigl(F_{0}^{*}\cong S \rTo^{
\begin{pmatrix}
 x_{1}\\ \vdots\\x_{c}
\end{pmatrix}
} S^{n*}\cong F_{1}^{*}\bigr) \cong M
$$
Thus the Tate resolution of $M$ is obtained by ``splicing'' together $\FF$ and $\FF^{*}$ via a map $\alpha: S\to S \cong S^{*}$ which may be taken to be multiplication by any generator of the socle. One well-known expression for such a generator is as $\det A$, where $A$ is a matrix expressing the $g_{i}$ as linear combinations of the $f_{i}$. Thus the Tate resolution of the residue field of $R$ has the form:
\begin{small}
 $$
\begin{diagram}[small]
\cdots&\uTo \\
\cdots&R^{c*}\otimes R^{c*}&\lTo&R^{c*}\\
\cdots&\uTo&&\uTo\\
\cdots&\bigwedge^{2}R^{c*} 
&\lTo& R^{c*}&\lTo^{
\begin{pmatrix}
 x_{1}\\ \vdots\\x_{c}
\end{pmatrix}
}
&R \\
&&&&&\uTo_{\alpha = \det A}  \\
&&&&&R&\lTo^{
\begin{pmatrix}
 x_{1}&\dots&x_{c}
\end{pmatrix}
}
&R^{c}
&\lTo& \bigwedge^{2}R^{c}&\cdots\\
&&&&&&&\uTo&&\uTo&\cdots\\
&&&&&&&R^{c}\otimes R&\lTo &R^{c}\otimes R^{c}&\cdots\\
&&&&&&&&&\uTo&\cdots
\end{diagram}
$$
\end{small}
\end{example}

The theory works with no essential change in the case when $R$ and $S$ are positively graded rings and also for any pair of complete intersections $I\subset J$ \emph{of the same codimension}.
In Theorem~\ref{CM approx for 2 reg seqs} we give an equally explicit description in the case when the codimensions of $I$ and $J$ are different.

\section{Duality in the Tate resolution of a Cohen-Macaulay module}

Tate resolutions associated with Cohen-Macaulay modules always have a sort of duality:

\begin{proposition}\label{general MCM approx} Let $R$ be a Gorenstein ring, and let $M$ be a Cohen-Macaulay $R$-module whose annihilator has codimension $m$ over $R$.
Let $(\FF,\delta)$ and $(\GG, \partial)$ be $R$-free resolutions of $M$ and of $M^{\vee}:=\Ext_{R}^{m}(M,R)$, with terms 
$F_{i}$ and $G_{i}$, respectively.
 There is a quasi-isomorphism $\phi: \FF \to \GG^{*}[-m]$: 
 $$
\begin{diagram}
 \cdots&\lTo^{\partial^{*}}& G^{*}_{m}&\lTo^{\partial^{*}} &G^{*}_{m-1}&\lTo^{\partial^{*}}&\cdots& \lTo^{\partial^{*}}&G^{*}_{0}&\lTo &0\\
 &&\uTo^{\phi_{0}}&&\uTo^{\phi_{1}}&&&&\uTo^{\phi_{m}}\\
 &&F_{0}&\lTo^{\delta}&F_{1}&\lTo^{\delta}&\cdots&\lTo^{\delta}&F_{m}&\lTo^{\delta}&\cdots
\end{diagram}.
 $$
The mapping cone $\MM(\phi)$, the total complex of the double complex above, is a Tate resolution for $M$, and $\MM(\phi^{*})$ is a Tate resolution of $M^{\vee}$.

In particular, the essential MCM approximation of $M$ over $R$ has a presentation
 $$
 F_{0}\oplus G_{m-1}^{*} \lTo^{
\begin{pmatrix}
 \delta&0\\
 \phi_{1}&\partial^{*}
\end{pmatrix}
}
F_{1}\oplus G_{m-2}^{*}.
 $$
\end{proposition}

\noindent{\bf Remark:} If we drop the Gorenstein hypothesis but still assume that $R$ is Cohen-Macaulay, and replace $(-)^{*}$ with $(-)^{\vee} = \Hom_{R}(-, \omega_{R})$, then similar statements still hold.

\begin{proof}Because $M$ and $M^{\vee}$ are Cohen-Macaulay modules of codimension $m$ we have $M \cong \Ext_{R}^{m}(M^{\vee},R) = \HH^{m}(\GG^{*})$, while
$\Ext^{i}_{R}(M^{\vee},R) = 0$ for $i\neq m$.

 The isomorphism $M = \HH_{0}(\FF) \cong \HH^{m}(\GG^{*})$ lifts to a map $F_{0}\to G^{* }_{m}$
 which induces a map $F_{1}\to G^{*}_{m-1}$, and thus to a  quasi-isomorphism
 $\phi: \FF\to \GG^{*}$ of cohomological degree $m$.
 It follows that the mapping cone $\MM(\phi)$ of $\phi$ has no homology. Since it coincides with $\FF$ in large
 homological degree, it is a Tate resolution of $M$.
 
Since the image of each map in $\MM(\phi)$ is a maximal Cohen-Macaulay module, every truncation of $\MM$ is a resolution of such a module, and thus the dual $\MM^{*}(\phi) = \MM(\phi)^{*}$ has no homology. It follows that $\phi^{*}$ is also a quasi-isomorphism. 
 \end{proof}
 
From Proposition~\ref{general MCM approx} we see that, beyond the free resolutions of $M$ and $M^{\vee}$, the new information in the Tate resolution lies in the description of the map of complexes $\phi$. The rest of this paper is devoted to further description of this maps.

In the situation of Proposition~\ref{general MCM approx}, suppose in addition that $R=S/I$. To construct a (generally non-minimal) $R$-free resolution $\FF$ of $M$ one might take an $S$-free resolution $\KK$ of $M$, tensor with $R$, and then extend it to an $R$-free resolution. The next result applies, in particular, to the case when $S$ is regular local and $R = S/I$ is Gorenstein,  and also to the case where $S$ is arbitrary and $I$ is generated by a regular sequence. It is a step toward building a map of complexes $\phi$ as in Proposition~\ref{general MCM approx}.

\def\EE{{\mathbb E}}
\begin{theorem}\label{sigma}
 Suppose that $S$ and $R = S/I$ are Noetherian rings and that $R$ has an
 $S$-free resolution $\EE$ 
 $$
 \EE: S \lTo^{\partial_{1}}E_{1}\lTo \cdots\lTo E_{c-1} \lTo^{\partial_{c}}E_{c}\lTo 0.
 $$ 
with 
$
 \partial_{c} \cong \partial_{1}^{*}.
$
Let $M$ be an $R$-module, and let $(\KK, \delta)$
be an $S$-free resolution of $M$. If $\sigma^{M}$ is a map of complexes 
$\EE\otimes \KK \to \KK$  with components
$$
\sigma^{M}_{i,j}: E_{i}\otimes K_{j} \to K_{i+j}
$$
that induces the multiplication map $R\otimes M\to M$ then
 $$
\begin{diagram}
 \cdots&\lTo^{R\otimes \delta}& R\otimes K_{c-1}&\lTo^{R\otimes \delta} &R\otimes K_{c}&\lTo^{R\otimes \delta}&R\otimes K_{c+1} &\lTo^{R\otimes \delta}& \cdots\\
&&&&\uTo^{R\otimes \sigma_{c,0}}&&\uTo^{R\otimes \sigma_{c,1}}&&\cdots\\
&& 0&\lTo&R\otimes K_{0}&\lTo^{R\otimes \delta}&R\otimes K_{1} &\lTo^{R\otimes \delta}&\cdots\\
\end{diagram}.
 $$
is a map of complexes inducing
an isomorphism $M = H_{0}(R\otimes \KK) \to H_{c}(R\otimes\KK) = \Tor_{c}^{S}(R,M)$.
\end{theorem}

\begin{lemma}\label{Tor} If $S$ and $R$ are as in Theorem~\ref{sigma}, then
the functor $\Tor^{S}_{c}(R, -)$, restricted to the category of $R$-modules, is equivalent to the identity functor.
\end{lemma}

\begin{proof}[Proof of Lemma~\ref{Tor}]
We compute $\Tor^{S}_{c}(R,M)$ using the given resolution of $R$. Because $\partial_{1}^{*}\otimes M = 0$ we see that
 $$
\Tor^{S}_{c}(R,M) = \ker(E_{c}\otimes M \rTo^{\partial_{1}^{*}\otimes M} E_{c-1}\otimes M) \cong M.
$$

Any choice of an isomorphism $E_{c}\cong S$ gives an equivalence between this functor and the identity functor.\end{proof}

\begin{proof}[Proof of Theorem~\ref{sigma}] We first prove that 
the maps $\sigma^{M}_{c,*}$ form a map of complexes $\sigma^{M}_{c}: R\otimes \KK \to R\otimes \KK[-c]$.
Write $\delta$ for the differential of $\KK$.
From the definition of the $\sigma^{M}_{i,j}$ we see that there are commutative diagrams
\begin{small}
$$
\begin{diagram}
R\otimes K_{i-1+c} &\lTo& R\otimes K_{i+c}\\
\uTo^{
R\otimes\begin{pmatrix}
 \sigma^{M}_{c,i-1} & \sigma^{M}_{c-1,i} 
\end{pmatrix}
}&& \uTo_{R\otimes\sigma^{M}_{ c,i}}\\
R\otimes \bigl((E_{c} \otimes K_{i-1})\oplus (E_{c-1}\otimes K_{i })\bigr) &
\lTo^{
R\otimes \begin{pmatrix}
\delta\otimes 1\\ \pm 1\otimes  \partial)
\end{pmatrix}}
&
R\otimes E_{c}\otimes \KK_{i}\\
\end{diagram}
$$
However, $R\otimes (1\otimes \partial): E_{c}\to E_{c-1}$ is 0, so the diagrams
$$
\begin{diagram}
R\otimes K_{i-1+c} &\lTo& R\otimes K_{i+c}\\
\uTo^{R\otimes \sigma^{M}_{c,i-1}}
&& \uTo_{R\otimes\sigma^{M}_{ c,i}}\\
R\otimes E_{c} \otimes K_{i-1} &
\lTo^{\delta\otimes 1}&
R\otimes E_{c}\otimes K_{i}\\
\end{diagram}
$$
\end{small}
also commute, as required.

We next show that for any $R$-module $M$ the map $\sigma_{c}^{M}$ induces a functorial isomorphism $M = \Tor^{S}_{0}(R,M)\to \Tor^{S}_{c}(R,M)$. We first prove functoriality.
 Let $\varphi:M\to N$ be a homomorphism of $R$-modules, let $\LL$ be the $S$-free resolution of $N$, and let $\phi: \KK\to\LL$ be a map extending $\varphi$. Choose maps $\sigma^{M}:\EE\otimes \KK \to \KK$ and $\sigma^{N}: \EE\otimes \LL \to \LL$ extending the multiplication maps as above.

There is a homotopy $\tau$ between the two compositions 
$$
\EE\otimes \KK \rTo^{\sigma^{N}\circ (1\otimes\phi)} \LL
$$
and 
$$
\EE\otimes \KK \rTo^{\phi\circ\sigma^{M}} \LL.
$$
because they cover the same map $R\otimes M \to N$ and $\LL$ is acyclic. Because $R\otimes \partial_{c} = 0$, this homotopy restricts to 
 a homotopy between the induced maps 
 $$
 R\otimes E_{c}\otimes \KK\rTo^{ R\otimes \bigl(\sigma^{N}\circ (1\otimes\phi)\bigr)}R\otimes \LL
 $$
 and 
 $$
 R\otimes E_{c}\otimes \KK\rTo^{ R\otimes(\phi\circ\sigma^{M})}R\otimes \LL.
 $$
In particular, the diagrams
$$
\begin{diagram}
 \Tor^{S}_{i+c}(R,M)& \rTo^{\Tor^{S}_{i+c}(R,\varphi)}& \Tor^{S}_{i+c}(R,N)\\
 \uTo^{\sigma^{M}_{*}} && \uTo_{\sigma^{N}_{*}}\\
 \Tor^{S}_{i}(R,M)& \rTo^{\Tor^{S}_{i}(R,\varphi)}& \Tor^{S}_{i}(R,N)\\
\end{diagram}
$$
commute. This proves the functoriality.

We next observe that  $\sigma^{R}_{c,0}$ is an isomorphism. This follows because we may choose
$\sigma^{R}: \EE \otimes \EE \to \EE$ to restrict to the identity map on the subcomplex $\EE\otimes E_{0}=\EE$. It follows from this that $\sigma_{c,0}^{R^{s}}$ is an isomorphism for any $s$.

From the right exact sequence 
$$
R\otimes K_{1} \to R\otimes K_{0} \to M\to 0
$$ 
we now get a commutative diagram
$$
\begin{diagram}
Tor_{c}^{S}(R,R\otimes K_{1})&\rTo& Tor_{c}^{S}(R,R\otimes K_{0}) &\rTo&  Tor_{c}^{S}(R,M)&\rTo& 0\\
\uTo^{\sigma_{*}^{K_{1}}}&&\uTo^{\sigma_{*}^{K_{0}}}&&\uTo^{\sigma_{*}^{M}}\\
Tor_{0}^{S}(R,R\otimes K_{1})&\rTo& Tor_{0}^{S}(R,R\otimes K_{0}) &\rTo&  Tor_{0}^{S}(R,M)&\rTo& 0\\
\end{diagram}.
$$
The bottom row is the $R$-free presentation of $M$, and the top row is also right exact because $\Tor^{S}_c(R,-)$ is an equivalence on the category of $R$-modules. The two left-hand  vertical maps are isomorphisms because $R\otimes K_{1}$ and $R\otimes K_{0}$ are free. It follows by a diagram chase that $\sigma_{*}^{M}$ is an isomorphism as well, 
completing the proof.
\end{proof}

If $R$ is a complete intersection in $S$ the maps $\sigma_{i,j}$ of Theorem~\ref{sigma} have a simpler description:

\begin{proposition}\label{compositions}
Suppose that $S$ is a Noetherian ring, that $\seq gc$ is a regular sequence in $S$, and that $R = S/(\seq gc)$, so that the $S$-free resolution of $R$ is the Koszul complex
$$
S\lTo^{\partial} S^{c}\lTo^{\partial}\bigwedge^{2}S^{c}\lTo^{\partial} \cdots\lTo^{\partial} \bigwedge^{c} S^{c}  \lTo 0.
$$ 
Suppose that $\KK$ is an $S$-free resolution of an $R$-module $M$, and, for $1\leq j\leq c$,  let $\tau_{j}: \KK \to \KK[1]$ be
a homotopy for multiplication by $g_{j}$ on $\KK$.  Let $e_{1},\dots, e_{c}$ be a basis for $S^{c}$ such that
$\partial(e_{i}) = g_{i}$. The map
$$
\sigma_{i,j}: \bigwedge^{i} S^{c} \otimes K_{j} \to K_{i+j}
$$
 that takes an  element
$e_{i_{1}}\wedge\cdots\wedge e_{i_{s}} \otimes a$  with $i_{1}<\cdots <i_{s}$ to
$\tau_{i_{1}}\circ \cdots \circ \tau_{i_{s}}(a)$ satisfies the properties of the maps $\sigma$ of Theorem~\ref{sigma}.

In particular, the map $\sigma_{c,i}: K_{i}\to K_{c+i}$
of Theorem~\ref{sigma}  may be chosen to be $\tau_{1}\circ\cdots\circ\tau_{c}$.
\end{proposition}

 Note that if $(\FF,\delta)$ is any complex and $\tau$ is a homotopy for multiplication by an element $g$, then
 $\delta$ anti-commutes with $\tau$ modulo $g$. Thus the order of the $\tau_{i}$ in the formula does not matter
 modulo $I$.

\begin{proof} Since the elements 
$e_{i_{1}}\wedge\cdots\wedge e_{i_{s}}$  with $i_{1}<\cdots <i_{s}$
form a basis for $\bigwedge^{s} S^{c}$, the maps $\sigma_{i,j}$ are well-defined.

Write $\delta$ for the differential of $\KK$. The differential of $\EE\otimes \KK$ acts on 
$E_{s}\otimes K_{j}$ as $\partial \otimes 1 + 1\otimes (-1)^{s} \delta$. 
From the defining property of the homotopies $\tau_{i}$ we have 
\begin{align*}
&\delta \sigma_{s,j} (e_{i_{1}}\wedge \cdots \wedge e_{i_{s}}\otimes a)\\
& = \delta\circ\tau_{i_{1}}\circ\cdots\circ\tau_{i_{s}}(a) 
= (g_{i_{1}}-\tau_{i_{1}}) \circ\delta\circ \tau_{i_{2}}\circ\cdots\circ\tau_{i_{s}}(a) = \cdots\\
&= \bigl(\sum_{j = 0}^{i-1} (-1)^{j} g_{i_{j}}\tau_{i_{1}}\circ\cdots\circ\widehat{\tau_{i_{i_{j}}}}\circ\dots \circ \tau_{i_{s}}(a)\bigr) 
+ (-1)^{s} \tau_{i_{1}}\circ\cdots\circ\tau_{i_{s}}\circ \delta(a)\\
& = \sigma_{s-1,j}( \partial (e_{i_{1}}\wedge \cdots \wedge e_{i_{s}})\otimes a)+\sigma_{i,j-1} (e_{i_{1}}\wedge \cdots \wedge e_{i_{s}}\otimes (-1)^{s}\delta (a))
 \end{align*}
 as required.
 \end{proof}
 
To apply Proposition~\ref{general MCM approx} using Theorem~\ref{sigma}, we will use the following result
in the case $N = M^{\vee}, N' = R$.

\begin{lemma}\label{inducing iso}
Suppose that $S$ is a Noetherian ring,  let $R = S/I$, and let $N, N'$ be a finitely generated $R$-modules. Let $\LL$ be an $S$-free resolution of $N$ and let $\GG$ be an $R$-free resolution of $N$, and let $\phi: \LL \to \GG$ be a map of complexes extending the identity map of $N$. If the depth of $J := ann N$ on $N'$ is $m$, then the map $H^{m}(\Hom_{R}(\GG,N')) \to H^{m}(\Hom_{S}(\LL,N'))$ is an isomorphism.
\end{lemma}

\noindent{\bf Remark:} It is well-known from duality theory that  
\begin{align*}
 H^{m}(\Hom_{R}(\GG,N')) &=  Ext_{R}^{m}(N,N') \\
 &\cong Ext_{S}^{m}(N,N') =  H^{m}(\Hom_{S}(\LL,N'));
\end{align*}
the point of the Lemma is that the comparison map $\phi$ induces the isomorphism, which doesn't seem to follow immediately from the standard proofs.

\begin{proof}
Suppose first that  $m=0$. Since $\phi$ induces the indentity on $N$ it induces the identity on
$\Hom_{R}(N,N') = \Hom_{S}(N, N')$.

We now induct on $m$, and we may suppose $m>0$. We may choose an element $x\in J$ that is a non-zerodivisor on $N'$, 
and consider the diagram
$$
\begin{diagram}[small]
0&\rTo &\Hom_{R}(\GG, N')&\rTo^{x}&  \Hom_{R}(\GG, N') &\rTo&  \Hom_{R}(\GG, N'/xN')& \rTo& 0\\
&&\dTo&&\dTo&&\dTo\\
0&\rTo &\Hom_{L}(\LL, N')&\rTo^{x}&  \Hom_{S}(\LL, N') &\rTo&  \Hom_{S}(\GG, N'/xN')& \rTo& 0\\
\end{diagram}
$$
Since $\Ext_{S}^{m-1}(N,N') = 0 = \Ext_{R}^{m-1}(N,N')$ while $x$ annihilates 
$\Ext_{S}^{m}(N,N') $ and $\Ext_{R}^{m}(N,N')$. Thus
we get a commutative diagram with exact rows
$$
\begin{diagram}[small]
 0&\rTo&H^{m-1}\Hom_{R}(\GG, N')&\rTo& H^{m}\Hom_{R}(\GG, N'/xN') &\rTo&0\\
 &&\dTo^{H^{m-1}\Hom(\phi, N')}&&\dTo^{H^{m}\Hom(\phi, N')}\\
0&\rTo&H^{m-1}\Hom_{S}(\LL, N')&\rTo& H^{m}\Hom_{S}(\LL, N'/xN') &\rTo&0\\
\end{diagram}
$$
and the left-hand vertical map is an isomorphism by induction.
\end{proof}

\section{The case of two regular sequences}\label{two regs}
We now fix the following {\bf NOTATION}: In this section, $S$ will denote a Gorenstein ring. We assume that $R = S/(g_{1}, \dots, g_c)$, where
$\seq gc$ is a regular sequence and that 
$M$ has the form $M = S/(f_1,\dots, f_{\cn})$, 
where $f_{1}, \dots, f_{\cn}$ is a regular sequence generating an ideal containing $g_{1}, \dots g_{c}$. We write
$\KK$ for the Koszul complex of $\seq f\cn$ over $S$.

We choose a matrix
 $A = (a_{i,j}): S^{c} \to S^{n}$ such that 
$$
A \begin{pmatrix}
  f_1\\ \vdots\\ f_{n}
\end{pmatrix}
= 
\begin{pmatrix}
 g_1\\ \vdots\\ g_{c}
\end{pmatrix};
$$
that is, $g_{j} = \sum_{i=1}^n a_{i,j}f_{i}$. We choose an identification of $\bigwedge^{c}S^{c}$ with $S$ and write $\alpha\in \bigwedge^{c}S^{n}$ for the image of 1
under $\bigwedge^{c}A$.
\fix{top exterior power of $S^n$ needed too, to identify the upper half with the dual of the lower half?}

We will give an explicit description the whole Tate resolution of $M$ over $R$ in terms of $\KK$ 
and $\alpha$ following the outline of Proposition~\ref{general MCM approx}. What allows us to do more
in this case than in the general case is the that we can choose the map $\sigma_{c, *}$ to kill all the higher homology
of $\KK\otimes R$.

First, we recall Tate's construction of the minimal $R$-free resolution $\FF$ of $M$. It is usually written as a complex whose underlying graded free module is the tensor product of $R\otimes \KK$ with $\D(R^{c})$, the divided power algebra on the free module
$R^{c}$, but for our purposes it will be useful to write it as the total complex of the double complex:
\begin{center}
\begin{tikzpicture}  
 [every node/.style={scale=.9},  auto]
 %Row 0
\node(0){$R$};
\node(1)[node distance = 1.15cm, right of=0]{$\bigwedge^{1}$};
\node(2)[node distance = 1.2cm, right of=1]{$\bigwedge^{2}$};
\node(3)[node distance = 1cm, right of=2]{};
\node(4)[node distance = 1cm, right of=3]{$\cdots$};
\node(5)[node distance = 1cm, right of=4]{};
\node(6)[node distance = 1.2cm, right of=5]{$\bigwedge^{k}$};
\node(7)[node distance = 1.2cm, right of=6]{};
\node(8)[node distance = .5cm, right of=7]{$\cdots$};
\node(9)[node distance = .5cm, right of=8]{};
\node(10)[node distance = 1.3cm, right of=9]{$\bigwedge^{\cn}$};
\node(11)[node distance = 2cm, right of=10]{$0$};
\node(12)[node distance = 1.4cm, right of=11]{};
%Label:
\node(label)[node distance=3cm, below of=1]{\large{$\FF:$}};
 %Row a
\node(a1)[node distance = 1cm, below of=1]{$\D_{1}$};
\node(a2)[node distance = 1cm, below of=2]{$\D_{1}\bigwedge^{2}$};
\node(a3)[node distance = 1cm, below of=3]{};
\node(a4)[node distance = 1cm, below of=4]{$\cdots$};
\node(a5)[node distance = 1cm, below of=5]{};
\node(a6)[node distance = 1cm, below of=6]{$\D_{1}\bigwedge^{k-1}$};
\node(a7)[node distance = 1cm, below of=7]{};
\node(a8)[node distance = 1cm, below of=8]{$\cdots$};
\node(a9)[node distance = 1cm, below of=9]{};
\node(a10)[node distance = 1cm, below of=10]{$\D_{1}\bigwedge^{\cn-1}$};
\node(a11)[node distance = 1cm, below of=11]{$\D_{1}\bigwedge^{\cn}$};
\node(a12)[node distance = 1cm, below of=12]{};
%Row b
\node(b2)[node distance = 1cm, below of=a2]{$\D_{2}$};
\node(b3)[node distance = 1cm, below of=a3]{};
\node(b4)[node distance = 1cm, below of=a4]{$\cdots$};
\node(b5)[node distance = 1cm, below of=a5]{};
\node(b6)[node distance = 1cm, below of=a6]{$\D_{2}\bigwedge^{k-2}$};
\node(b7)[node distance = 1cm, below of=a7]{};
\node(b8)[node distance = 1cm, below of=a8]{$\cdots$};
\node(b9)[node distance = 1cm, below of=a9]{};
\node(b10)[node distance = 1cm, below of=a10]{$\D_{2}\bigwedge^{\cn-2}$};
\node(b11)[node distance = 1cm, below of=a11]{$\D_{2}\bigwedge^{\cn-1}$};
\node(b12)[node distance = 1cm, below of=a12]{};
%between b,c
%Row b+
\node(b+3)[node distance = 1.5cm, below of=b3]{};
\node(b+4)[node distance = 1.5cm, below of=b4]{};
\node(b+5)[node distance = 1.5cm, below of=b5]{};
\node(b+6)[node distance = 1.5cm, below of=b6]{$\vdots$};
\node(b+7)[node distance = 1.5cm, below of=b7]{};
\node(b+8)[node distance = 1.5cm, below of=b8]{$\cdots$};
\node(b+9)[node distance = 1.5cm, below of=b9]{};
\node(b+10)[node distance = 1.5cm, below of=b10]{$\vdots$};
\node(b+11)[node distance = 1.5cm, below of=b11]{$\cdots$};
%Row c
\node(c6)[node distance = 3.2cm, below of=b6]{$\D_{k}\bigwedge^{k-1}$};
\node(c7)[node distance = 3.2cm, below of=b7]{};
\node(c8)[node distance = 3.2cm, below of=b8]{$\cdots$};
\node(c9)[node distance = 3.2cm, below of=b9]{};
\node(c10)[node distance = 3.2cm, below of=b10]{$\D_{k}\bigwedge^{\cn-k}$};
\node(c11)[node distance = 3.2cm, below of=b11]{$\D_{2}\bigwedge^{\cn-k+1}$};
\node(c12)[node distance = 3.2cm, below of=b12]{};
%Row d
\node(d7)[node distance = .9cm, below of=c7]{};
\node(d8)[node distance = .9cm, below of=c8]{$\cdots$};
\node(cd9)[node distance = .9cm, below of=c9]{};
\node(d10)[node distance = .9cm, below of=c10]{$\vdots$};
\node(d11)[node distance = .9cm, below of=c11]{$\cdots$};

%%%%%%%%%
%horizontal arrows
%Row 0
\draw[<-](0)to node{} (1);
\draw[<-](1)to node{} (2);
\draw[<-](2)to node{} (3);
\draw[<-](5)to node{} (6);
\draw[<-](6)to node{} (7);
\draw[<-](9)to node{} (10);
\draw[<-](10)to node{} (11);
\draw[<-](11)to node{} (12);
%Row a
\draw[<-](a1)to node{} (a2);
\draw[<-](a2)to node{} (a3);
\draw[<-](a5)to node{} (a6);
\draw[<-](a6)to node{} (a7);
\draw[<-](a9)to node{} (a10);
\draw[<-](a10)to node{} (a11);
\draw[<-](a11)to node{} (a12);
%Row b
\draw[<-](b2)to node{} (b3);
\draw[<-](b5)to node{} (b6);
\draw[<-](b6)to node{} (b7);
\draw[<-](b9)to node{} (b10);
\draw[<-](b10)to node{} (b11);
\draw[<-](b11)to node{} (b12);
%Rowc
%\draw[<-](c2)to node{} (c3);
%\draw[<-](c5)to node{} (c6);
\draw[<-](c6)to node{} (c7);
\draw[<-](c9)to node{} (c10);
\draw[<-](c10)to node{} (c11);
\draw[<-](c11)to node{} (c12);
%vertical arrows
%a to 0
\draw [style=dashed] (a1) to node{} (0);
\draw[->] (a1) to node{} (1);
\draw[->] (a2) to node{} (2);
\draw[->] (a6) to node{} (6);
\draw[->] (a10) to node{} (10);
%b to a
\draw [style=dashed] (b2) to node{} (a1);
\draw[->] (b2) to node{} (a2);
\draw[->] (b6) to node{} (a6);
\draw[->] (b10) to node{} (a10);
%c to b
\draw [style=dashed] (c6) to node{} (b2);
%cd to c
\draw [style=dashed] (d7) to node{} (c6); dt
\end{tikzpicture}
\end{center}
where for compactness we have written $\bigwedge^{i}$ for $R\otimes \bigwedge^{i}S^{\cn}$ and $\D_{i}$ for the $i$-th divided power of $R^{c}$ and suppressed the tensor product signs, so that for example we have written
$\D_{2}\bigwedge^{k}$ in place of $\D_{2}R^{c}\otimes_{S}\bigwedge^{k}R^{\cn}$.

By Proposition~\ref{general MCM approx} the Tate resolution of $M$ is the mapping cone of any map of complexes
$\phi: \FF\to \FF^*[-c]$ that induces an isomorphism 
$$
M= H_{0}(\FF) \to H_{0}(\FF^{*}[-c]) = H_{c}(\FF^{*}) \cong M.
$$
 In this case the maps  $\sigma_{c,i}:  K_{i} \to K_{i+c}$ of
Theorem~\ref{sigma} take a simple form. (A similar result holds for all the $\sigma_{i,j}$, but we do not need this.)

\begin{proposition}\label{comparison for Koszul complexes} Let $e'_{1},\dots e'_{n}$ be a basis of $S^{n}$, and let $\delta_{1}: S^{n} \to S$ send $e'_{i}\to f_{i}$. Let $(\KK,\delta)$ be the Koszul complex
$$
\KK: S\lTo^{\delta_{1}}S^{n}\lTo^{\delta_{2}}\bigwedge^{2}S^{n}\lTo \cdots
$$
The homotopy for an element $g = \sum_{i} a_{i}f_{i}$ on $\KK$ is exterior multiplication by $\sum_ia_{i}e'_{i}$. Thus if
$A = (a_{i,j}): S^{c} \to S^{n}$ is a $c\times n$ matrix such that 
$$
A \begin{pmatrix}
  f_1\\ \vdots\\ f_{n}
\end{pmatrix}
= 
\begin{pmatrix}
 g_1\\ \vdots\\ g_{c}
\end{pmatrix}
$$
then $\sigma_{c,0}: \KK\to \KK[-c]$ may be taken to be exterior multiplication by 
the image in $\bigwedge^{c}S^{n}$ of a generator $e_{1}\wedge\cdots\wedge e_{c}$
 of $\bigwedge^{c}S^{c}$ under the map 
$$
S \cong \bigwedge^{c}S^{c} \rTo^{\bigwedge^{c}A}\bigwedge^{c}S^{n}.
$$
\end{proposition}
\begin{proof}
It is easy to check directly that a homotopy for $f_{i}$ on $\KK$ is exterior multiplication by $e'_{i}$. The given formula for a homotopy for $g$ follows by linearity. 

By Proposition~\ref{compositions} the maps $\sigma_{i,j}$ are defined by compositions of the homotopies
$\tau_{j}$ for the $g_{j}$ on $\KK$, and the composition  $\tau_{1}\circ\cdots\circ\tau_{c}$ is thus
exterior multiplication by $\bigwedge^{c}A(e_{1}\wedge\cdots\wedge e_{c})$, as claimed.
\end{proof}

\begin{theorem}\label{CM approx for 2 reg seqs}
With notation as above, let $\KK$ be the Koszul complex resolving $M$ over $S$ and let $\FF$ be the Eisenbud-Shamash resolution of $M$ over $R$. Let $\phi': R\otimes \KK \to R\otimes \KK[-c] \cong R\otimes\KK^{*}[c]$ be the composition of the map defined in Proposition~\ref{comparison for Koszul complexes} with the isomorphism
induced by a choice of isomorphism $\beta: \bigwedge^{n}\S^{n} \to S$. Let $\phi: \FF \to \FF^{*}[c]$ be the composition
$$
\FF\rTo^{\pi} R\otimes \KK \rTo^{\phi'}R\otimes\KK^{*}[c] \rTo^{\pi^{*}} \FF^{*}[c].
$$
where $\pi$ is the projection with kernel $\oplus_{i\geq 1}\D_{i}(R^{c)}\otimes \KK[-i]$
The map $\phi$ is a homomorphism of complexes and maps $M = H_{0}(\FF)$ isomorphically to
$H_{0}(\FF^{*}[c])$. Thus the mapping cone $\MM(\phi)$ of $\phi$ is a Tate resolution of $M$ over $R$. If 
$I\subset \mm J$, then this Tate resolution is minimal.
\end{theorem}

Note that the complex $\KK$ is a subcomplex of $\FF$, \emph{not} a quotient complex, and thus $\pi$
and $\pi^{*}$ are \emph{not} maps of complexes. Nevertheless, the Theorem asserts that
$\phi$ is a map of complexes.

\begin{proof}[Proof of Theorem~\ref{CM approx for 2 reg seqs}]
Consider the doubly infinite diagram whose $i$-th and $i+1$-st columns are:
\begin{small}
 $$
\begin{diagram}[small]
&&\uTo&&\uTo\\
\cdots&\lTo&\S_{1}(R^{c})\otimes\bigwedge^{i+c+1}R^{n}&\lTo&\S_{1}(R^{c*})\otimes\bigwedge^{i+c+2}R^{n}&\lTo&\cdots\\
\cdots&&\uTo&&\uTo&&\cdots\\
\cdots&\lTo&\bigwedge^{i+c}R^{n}&\lTo&\bigwedge^{i+c+1}R^{n}&\lTo&\cdots\\
&&\uTo^{(-1)^{i}\phi'_{i}}&&\uTo^{(-1)^{i+1}\phi'_{i+1}}\\
\cdots&\lTo&\bigwedge^{i}R^{n}&\lTo&\bigwedge^{i+1}R^{n}&\lTo&\cdots\\
\cdots&&\uTo&&\uTo&&\cdots\\
\cdots&\lTo&\D_{1}(R^{c})\otimes\bigwedge^{i-1}R^{n}&\lTo&\D_{1}(R^{c})\otimes\bigwedge^{i}R^{n}&\lTo&\cdots\\
\end{diagram}
$$
\end{small}
with the term $\bigwedge^{i}R^{n}$ in position $(i,0)$. The maps in the bottom two rows of the diagram
are those of the $R$ free Eisenbud-Shamash resolution of $S/(\seq fn)$. Using the isomorphism $\beta$ we may identify $\bigwedge^{j}R^{n}$ with $\bigwedge^{n-j}(R^{n*})$, and with this identification, taking into account that $(\D_{i}R^{c})^{*}$ is is naturally isomorphic to $\S_{i}(R^{c*})$, the upper two rows of the diagram are isomorphic to the dual of the lower two rows, shifted $c$ steps to the left. Thus each row is itself a complex and the squares in the lower two rows, and dually in the upper two rows, commute up to sign.

We claim that, with the map $\phi'$ between the two middle rows,  the diagram is a double complex: that is, the squares in the middle two rows commute up to sign, and the vertical maps as well as the horizontal ones compose to zero. 

The lower of the middle two rows is simply the Koszul complex of $\seq fn$, and the upper of the middle two rows is the same Koszul complex, shifted $c$ steps to the left. We have already shown in Proposition~\ref{comparison for Koszul complexes} that the maps $\phi_{i}$ commute with the differentials of the these Koszul complexes.

We must still show that the composition of consecutive vertical maps is 0. But the columns of the diagram are exactly the complexes first described in \cite[Section 2]{BE} and \cite{Ki}, and given an exposition in \cite[Appendix A.2.10]{E}. (See also the more conceptual construction in \cite{W}, which follows ideas of \cite{Ke}.) 

However, as this is the only fact about the vertical columns that we need, it seems worth pointing out that the result is elementary, a direct extension of ``Cramer's rule'' for solving linear equations: since the whole diagram is self-dual, it may be reduced to showing that the composition
$$ 
 \D_{1}(R^{c})\otimes \bigwedge^{i-1}R^{n} = R^{c} \otimes \bigwedge^{i-1}R^{n} 
 \to   \bigwedge^{i} R^{n}
 \rTo^{\phi'} \bigwedge^{i+c} R^{n}
$$
is zero, and direct computation shows that the components of this map are the $(c+1)\times (c+1)$ minors of the 
matrix derived from $A$ by repeating a row.

Finally, we must show that the composed map of complexes $\pi\circ \phi'\circ \pi$
induces a isomorphism $H_{0}(\FF)\to H_{0}(\FF^{*}[c]) = H_{c}(\FF^{*})$.
To this end, consider the maps of complexes
$$
R\otimes\KK \rInto^{\iota} \FF \rTo^{\pi\circ \phi'\circ \pi} \FF^{*}[c] \rOnto^{\iota^{*}} R\otimes\KK^{*}[c]
$$
where $\iota$ is the natural inclusion of complexes. It is obvious that
$\iota$ induces an isomorphism $M = H_{0}(R\otimes\KK) \to H_{0}(\FF)$. Lemma~\ref{inducing iso} shows that
$\iota^{*}$ induces an isomorphism $H_{0}(\FF^{*}[c])= H_{c}(\FF^{*})\to H_{c}(R\otimes\KK^{*})= H_{0}(R\otimes\KK^{*}[c])$.
Finally, the composition $\iota^{*}\circ (\pi\circ \phi'\circ \pi) \circ \iota$ is just $\phi'$ composed with the
isomorphism $R\otimes\KK[-c] \cong R\otimes\KK^{*}[c]$ induced by $\beta$. This induces an isomorphism
$H_{0}(R\otimes\KK) \to H_{0}(R\otimes\KK^{*}[c])$ by Proposition~\ref{comparison for Koszul complexes}. Thus 
$$
\pi\circ \phi'\circ \pi: H_{0}\FF \to H_{0}(\FF^{*}[c])
$$
is an isomorphism as well, completing the proof. 
\end{proof}
Figure~\ref{Figure 2} shows the Tate Resolution of $S/(\seq f\cn)$ as an $R = S/(\seq gc)$-module, in case $\cn-c$ is odd. The bold arrows are given by wedge product with $\alpha$. The columns  are the complexes  ${\mathcal C}_{i}$ that appear in Figure A2.6 of \cite{E}. The dashed line passes through the terms of homological degree 0.
\begin{center}
\begin{tikzpicture}[thick,scale=0.55, every node/.style={transform shape}]
%%%%%%%Node naming conventions: 
%middle row is 0
%rows below middle are a,b,c,...
%rows above middle are a',b',\cn
%cols from center to the right are 0,1,2...
%to the left are m1,m2... , ie -1,-2...

 %%%%%%%%%%%%Below the middle
 %Row 0
\node(0){$R$};
 \node(m1)[node distance = 2.8cm, left of=0]{0}; %node ``-1''
\node(1)[node distance = 2.8cm, right of=0]{$\bigwedge^{1}$};
\node(2)[node distance = 2.9cm, right of=1]{$\bigwedge^{2}$};
\node(3)[node distance = 1.8cm, right of=2]{};
\node(4)[node distance = 1cm, right of=3]{$\cdots$};
\node(5)[node distance = .2cm, right of=4]{};
\node(6)[node distance = 1.5cm, right of=5]{$\bigwedge^{k}$};
\node(7)[node distance = 1.4cm, right of=6]{};
\node(8)[node distance = .4cm, right of=7]{$\cdots$};
\node(9)[node distance = .2cm, right of=8]{};
\node(95)[node distance = 1cm, right of=9]{$\cdots$};
\node(10)[node distance = 1.7cm, right of=95]{$\bigwedge^{\cn-c}$};
\node(11)[node distance = 2.4cm, right of=10]{$\bigwedge^{\cn-c+1}$};
\node(12)[node distance = 1.6cm, right of=11]{};

 %Row a
\node(am1)[node distance = 1cm, below of=m1]{$$};
\node(a0)[node distance = 1cm, below of=0]{$0$};
\node(a1)[node distance = 1cm, below of=1]{$\D_{1}$};
\node(a2)[node distance = 1cm, below of=2]{$\D_{1}\bigwedge^{1}$};
\node(a3)[node distance = 1cm, below of=3]{};
\node(a4)[node distance = 1cm, below of=4]{$\cdots$};
\node(a5)[node distance = 1cm, below of=5]{};
\node(a6)[node distance = 1cm, below of=6]{$\D_{1}\bigwedge^{k-1}$};
\node(a7)[node distance = 1cm, below of=7]{};
\node(a8)[node distance = 1cm, below of=8]{$\cdots$};
\node(a9)[node distance = 1cm, below of=9]{};
\node(a95)[node distance = 1cm, below of=95]{$\cdots$};
\node(a10)[node distance = 1cm, below of=10]{$\D_{1}\bigwedge^{\cn-c-1}$};
\node(a11)[node distance = 1cm, below of=11]{$\D_{1}\bigwedge^{\cn-c}$};
\node(a12)[node distance = 1cm, below of=12]{};

%Row b
\node(b1)[node distance = 1cm, below of=a1]{0};
\node(b2)[node distance = 1cm, below of=a2]{$\D_{2}$};
\node(b3)[node distance = 1cm, below of=a3]{};
\node(b4)[node distance = 1cm, below of=a4]{$\cdots$};
\node(b5)[node distance = 1cm, below of=a5]{};
\node(b6)[node distance = 1cm, below of=a6]{$\D_{2}\bigwedge^{k-2}$};
\node(b7)[node distance = 1cm, below of=a7]{};
\node(b8)[node distance = 1cm, below of=a8]{$\cdots$};
\node(b9)[node distance = 1cm, below of=a9]{};
\node(b95)[node distance = 1cm, below of=a95]{$\cdots$};
\node(b10)[node distance = 1cm, below of=a10]{$\D_{2}\bigwedge^{\cn-c-2}$};
\node(b11)[node distance = 1cm, below of=a11]{$\D_{2}\bigwedge^{\cn-c-1}$};
\node(b12)[node distance = 1cm, below of=a12]{};

%Row c
\node(c3)[node distance = 1cm, below of=b3]{};
\node(c4)[node distance = 1cm, below of=b4]{};
\node(c5)[node distance = 1cm, below of=b5]{};
\node(c6)[node distance = 1cm, below of=b6]{$\vdots$};
\node(c7)[node distance = 1cm, below of=b7]{};
\node(c8)[node distance = 1cm, below of=b8]{$\cdots$};
\node(c9)[node distance = 1cm, below of=b9]{};
\node(c95)[node distance = 1cm, below of=b95]{};
\node(c10)[node distance = 1cm, below of=b10]{$\vdots$};
\node(c11)[node distance = 1cm, below of=b11]{$\cdots$};

%Row d
%\node(b5)[node distance = 1cm, below of=a5]{};
\node(d5)[node distance = 2.1cm, below of=b5]{0};
\node(d6)[node distance = 2.1cm, below of=b6]{$\D_{k}$};
\node(d7)[node distance = 2.1cm, below of=b7]{};
\node(d8)[node distance = 2.1cm, below of=b8]{$\cdots$};
\node(d9)[node distance = 2.1cm, below of=b9]{};
\node(d95)[node distance = 2.1cm, below of=b95]{$\cdots$};
\node(d10)[node distance = 2.1cm, below of=b10]{$\D_{k}\bigwedge^{\cn-c-k}$};
\node(d11)[node distance = 2.1cm, below of=b11]{$\D_{2}\bigwedge^{\cn-c-k+1}$};
\node(d12)[node distance = 2.1cm, below of=b12]{};

%Row e
\node(e7)[node distance = .9cm, below of=d7]{};
\node(e8)[node distance = .9cm, below of=d8]{$\ddots$};
\node(e9)[node distance = .9cm, below of=d9]{};
\node(e95)[node distance = .9cm, below of=d95]{$\ddots$};
\node(e10)[node distance = .9cm, below of=d10]{$\ddots$};
\node(e11)[node distance = .9cm, below of=d11]{$\ddots$};

%%%%%%%%%
%horizontal arrows
%Row 0
\draw[<-](m1)to node{} (0);
\draw[<-](0)to node{} (1);
\draw[<-](1)to node{} (2);
\draw[<-](2)to node{} (3);
\draw[<-](5)to node{} (6);
\draw[<-](6)to node{} (7);
\draw[<-](95)to node{} (10);
\draw[<-](10)to node{} (11);
\draw[<-](11)to node{} (12);

%Row a
%\draw[<-](am1)to node{} (a0);
\draw[<-](a0)to node{} (a1);
\draw[<-](a1)to node{} (a2);
\draw[<-](a2)to node{} (a3);
\draw[<-](a5)to node{} (a6);
\draw[<-](a6)to node{} (a7);
\draw[<-](a95)to node{} (a10);
\draw[<-](a10)to node{} (a11);
\draw[<-](a11)to node{} (a12);

%Row b
\draw[<-](b1)to node{} (b2);
\draw[<-](b2)to node{} (b3);
\draw[<-](b5)to node{} (b6);
\draw[<-](b6)to node{} (b7);
\draw[<-](b95)to node{} (b10);
\draw[<-](b10)to node{} (b11);
\draw[<-](b11)to node{} (b12);

%Row d
\draw[<-](d5)to node{} (d6);
\draw[<-](d6)to node{} (d7);
\draw[<-](d95)to node{} (d10);
\draw[<-](d10)to node{} (d11);
\draw[<-](d11)to node{} (d12);
%%

%%%%%%%%%%%%Above the middle
 %Row a'
\node(a'0)[node distance = 1.15cm, above of=0]{$\bigwedge^{c}$};
\node(a'1)[node distance = 1.15cm, above of=1]{$\bigwedge^{c+1}$};
\node(a'2)[node distance = 1.15cm, above of=2]{$\bigwedge^{c+2}$};
\node(a'3)[node distance = 1.15cm, above of=3]{};
\node(a'4)[node distance = 1.15cm, above of=4]{$\cdots$};
\node(a'5)[node distance = 1.15cm, above of=5]{};
\node(a'6)[node distance = 1.15cm, above of=6]{$\bigwedge^{c+k}$};
\node(a'7)[node distance = 1.15cm, above of=7]{};
\node(a'8)[node distance = 1.15cm, above of=8]{$\cdots$};
\node(a'9)[node distance = 1.15cm, above of=9]{};
\node(a'95)[node distance = 1.15cm, above of=95]{$\bigwedge^{\cn-1}$};
\node(a'10)[node distance = 1.15cm, above of=10]{$\bigwedge^{\cn}$};
\node(a'11)[node distance = 1.15cm, above of=11]{0};
\node(a'12)[node distance = 1.15cm, above of=12]{};

\node(a'm1)[node distance = 1.15cm, above of=m1]{$\bigwedge^{c-1}$};
%\node(a'm2)[node distance = 1.15cm, left of=a'm1]{$\bigwedge^{c-2}$};

%Row b'
\node(b'm1)[node distance = 1cm, above of=a'm1]{$\S_{1}\bigwedge^{c}$};
\node(b'0)[node distance = 1cm, above of=a'0]{$\S_{1}\bigwedge^{c+1}$};
\node(b'1)[node distance = 1cm, above of=a'1]{$\S_{1}\bigwedge^{c+2}$};
\node(b'2)[node distance = 1cm, above of=a'2]{$\S_{1}\bigwedge^{c+3}$};
\node(b'3)[node distance = 1cm, above of=a'3]{};
\node(b'4)[node distance = 1cm, above of=a'4]{$\cdots$};
\node(b'5)[node distance = 1cm, above of=a'5]{$$};
\node(b'6)[node distance = 1cm, above of=a'6]{$\S_{1}\bigwedge^{c+k+1}$};
\node(b'7)[node distance = 1cm, above of=a'7]{};
\node(b'8)[node distance = 1cm, above of=a'8]{$\cdots$};
\node(b'9)[node distance = 1cm, above of=a'9]{};
\node(b'95)[node distance = 1cm, above of=a'95]{$\S_{1}\bigwedge^{\cn}$};
\node(b'10)[node distance = 1cm, above of=a'10]{$0$};
c'
%Row c'
\node(c'm1)[node distance = 1cm, above of=b'm1]{$\vdots$};
\node(c'0)[node distance = 1cm, above of=b'0]{$\vdots$};
\node(c'1)[node distance = 1cm, above of=b'1]{$\vdots$};
\node(c'2)[node distance = 1cm, above of=b'2]{$\vdots$};
\node(c'3)[node distance = 1cm, above of=b'3]{};
\node(c'4)[node distance = 1cm, above of=b'4]{};
\node(c'6)[node distance = 1cm, above of=b'6]{$\vdots$};
\node(c'7)[node distance = 1cm, above of=b'7]{};
\node(c'8)[node distance = 1cm, above of=b'8]{};

%Row d'
\node(d'm1)[node distance = 2.1cm, above of=b'm1]{$\S_{\cn-c-k}\bigwedge^{\cn-k-1}$};
\node(d'0)[node distance = 2.1cm, above of=b'0]{$\S_{\cn-c-k}\bigwedge^{\cn-k}$};
\node(d'1)[node distance = 2.1cm, above of=b'1]{$\S_{\cn-c-k}\bigwedge^{\cn-k+1}$};
\node(d'2)[node distance = 2.1cm, above of=b'2]{$\S_{\cn-c-k}\bigwedge^{\cn-k+2}$};
\node(d'3)[node distance = 2.1cm, above of=b'3]{};
\node(d'4)[node distance = 2.1cm, above of=b'4]{$\cdots$};
\node(d'5)[node distance = 2.1cm, above of=b'5]{$$};

\node(d'6)[node distance = 2.1cm, above of=b'6]{$\S_{\cn-c-k}\bigwedge^{\cn}$};
\node(d'7)[node distance =2.1cm, above of=b'8]{0};
%Row e'
\node(e'm1)[node distance = .5cm, above of=d'm1]{$\ddots$};
\node(e'0)[node distance = .5cm, above of=d'0]{$\ddots$};
\node(e'1)[node distance = .5cm, above of=d'1]{$\ddots$};
%\node(e'5)[node distance = .5cm, above of=d'5]{$\ddots$};
\node(e'5)[node distance = .5cm, above of=d'5]{$\ddots$};

%%Arrows:
%Row a'
\draw[<-](a'm1)to node{} (a'0);
\draw[<-](a'0)to node{} (a'1);
\draw[<-](a'1)to node{} (a'2);
\draw[<-](a'2)to node{} (a'3);
\draw[<-](a'5)to node{} (a'6);
\draw[<-](a'6)to node{} (a'7);
\draw[<-](a'95)to node{} (a'10);
\draw[<-](a'10)to node{} (a'11);

%Row b'
\draw[<-](b'm1)to node{} (b'0);
\draw[<-](b'0)to node{} (b'1);
\draw[<-](b'1)to node{} (b'2);
\draw[<-](b'2)to node{} (b'3);
\draw[<-](b'5)to node{} (b'6);
\draw[<-](b'6)to node{} (b'7);
\draw[<-](b'95)to node{} (b'10);

%row d'
\draw[<-](d'm1)to node{} (d'0);
\draw[<-](d'0)to node{} (d'1);
\draw[<-](d'1)to node{} (d'2);
\draw[<-](d'2)to node{} (d'3);
%\draw[<-](d'5)to node{} (d'6);
%\draw[<-](d'6)to node{} (d'7);

%Row \cn
\draw[<-](d'6)to node{} (d'7);

%vertical arrows
%a to 0
%\draw [style=dashed] (am1) to node{} (0);
%\draw[->] (am1) to node{} (m1);
%\draw[->] (a0) to node{} (0);
\draw[->] (a1) to node{} (1);
\draw[->] (a2) to node{} (2);
\draw[->] (a6) to node{} (6);
\draw[->] (a10) to node{} (10);

%b to a
%\draw [style=dashed] (b2) to node{} (a1);
\draw[->] (b2) to node{} (a2);
\draw[->] (b6) to node{} (a6);
\draw[->] (b10) to node{} (a10);

%c to b
%\draw [style=dashed] (d6) to node{} (b2);

%d to c
%\draw [style=dashed] (e7) to node{} (d6); dt

%0 to a'
\draw [style=dashed] (a'1) to node{} (0);
%\draw[<-] (a'm1) [style=ultra thick] to node{} (m1);
\draw[<-] (a'0) [style=ultra thick] to node{} (0);
\draw[<-] (a'1) [style=ultra thick] to node{} (1);
\draw[<-] (a'2) [style=ultra thick]to node{} (2);
\draw[<-] (a'6) [style=ultra thick]to node{} (6);
\draw[<-] (a'10)[style=ultra thick]to node{} (10);

%a' to b'
\draw [style=dashed] (b'2) to node{} (a'1);
\draw[<-] (b'm1) to node{} (a'm1);
\draw[<-] (b'0) to node{} (a'0);
\draw[<-] (b'1) to node{} (a'1);
\draw[<-] (b'2) to node{} (a'2);
\draw[<-] (b'6) to node{} (a'6);
\draw[<-] (b'95) to node{} (a'95);

%b' to d'
\draw [style=dashed] (d'6) to node{} (b'2);
%d' to e'
%\draw [style=dashed] (e'7) to node{} (d'6);
\end{tikzpicture}\label{Figure 2}
\vskip -.25cm
{\bf Figure~\ref{Figure 2}}
\end{center}

Here the entries of the matrices represented by horizontal arrows are the $f_{i}$, and the entries of the 
matrices represented by the vertical arrows are the $g_{j}$, except for the bold arrows, which where the
entries are the $c\times c$ minors of $A$.

The part of the complex represented by the lower half of the diagram
is infinite, and each row has is a tensor product of a $\D_{k}$ with the Koszul complex on $\seq f\cn$. 
Each row is a tensor product of a $\D_{k}$ or an $\S_{k}$ with the Koszul complex on $\seq f\cn$.
The columns, on the other hand, are the complexes  ${\mathcal C}_{i}$ that appear in Figure A2.6 of \cite{E}.

\begin{corollary}\label{explicit MCM}
The minimal free resolution of the maximal Cohen-Macaulay approximation $M'$ of $M$ has the form shown in Figure~\ref{Figure3}. Thus $M'$ has no free summand, and requires
$$
1+ \sum_{1\leq i\leq (n-c-1)/2}{n\choose c+1+2i}{c-1+i\choose i}
$$
generators.
\end{corollary}
\vfill\eject

 \begin{center}
\begin{tikzpicture}[thick,scale=0.8, every node/.style={transform shape}]
 %%%%%%%%%%%%Below the middle
 %Row 0
\node(0){$R$};
\node(1)[node distance = 1.15cm, right of=0]{$\bigwedge^{1}$};
\node(2)[node distance = 1.2cm, right of=1]{$\bigwedge^{2}$};
\node(3)[node distance = 1cm, right of=2]{};
\node(4)[node distance = 1cm, right of=3]{$\cdots$};
\node(5)[node distance = 1cm, right of=4]{};
\node(6)[node distance = 1.2cm, right of=5]{$\bigwedge^{k}$};
\node(7)[node distance = 1.2cm, right of=6]{};
\node(8)[node distance = .5cm, right of=7]{$\cdots$};
\node(9)[node distance = .5cm, right of=8]{};
\node(95)[node distance = 1cm, right of=9]{$\cdots$};
\node(10)[node distance = 1.7cm, right of=95]{$\bigwedge^{\cn-c}$};
\node(11)[node distance = 2.2cm, right of=10]{$\bigwedge^{\cn-c+1}$};
\node(12)[node distance = 1.4cm, right of=11]{};
 
 %Row a
\node(a1)[node distance = 1cm, below of=1]{$\D_{1}$};
\node(a2)[node distance = 1cm, below of=2]{$\D_{1}\bigwedge^{1}$};
\node(a3)[node distance = 1cm, below of=3]{};
\node(a4)[node distance = 1cm, below of=4]{$\cdots$};
\node(a5)[node distance = 1cm, below of=5]{};
\node(a6)[node distance = 1cm, below of=6]{$\D_{1}\bigwedge^{k-1}$};
\node(a7)[node distance = 1cm, below of=7]{};
\node(a8)[node distance = 1cm, below of=8]{$\cdots$};
\node(a9)[node distance = 1cm, below of=9]{};
\node(a95)[node distance = 1cm, below of=95]{$\cdots$};
\node(a10)[node distance = 1cm, below of=10]{$\D_{1}\bigwedge^{\cn-c-1}$};
\node(a11)[node distance = 1cm, below of=11]{$\D_{1}\bigwedge^{\cn-c}$};
\node(a12)[node distance = 1cm, below of=12]{};

%Row b
\node(b2)[node distance = 1cm, below of=a2]{$\D_{2}$};
\node(b3)[node distance = 1cm, below of=a3]{};
\node(b4)[node distance = 1cm, below of=a4]{$\cdots$};
\node(b5)[node distance = 1cm, below of=a5]{};
\node(b6)[node distance = 1cm, below of=a6]{$\D_{2}\bigwedge^{k-2}$};
\node(b7)[node distance = 1cm, below of=a7]{};
\node(b8)[node distance = 1cm, below of=a8]{$\cdots$};
\node(b9)[node distance = 1cm, below of=a9]{};
\node(b95)[node distance = 1cm, below of=a95]{$\cdots$};
\node(b10)[node distance = 1cm, below of=a10]{$\D_{2}\bigwedge^{\cn-c-2}$};
\node(b11)[node distance = 1cm, below of=a11]{$\D_{2}\bigwedge^{\cn-c-1}$};
\node(b12)[node distance = 1cm, below of=a12]{};

%Row c
\node(c3)[node distance = 1.5cm, below of=b3]{};
\node(c4)[node distance = 1.5cm, below of=b4]{};
\node(c5)[node distance = 1.5cm, below of=b5]{};
\node(c6)[node distance = 1.5cm, below of=b6]{$\vdots$};
\node(c7)[node distance = 1.5cm, below of=b7]{};
\node(c8)[node distance = 1.5cm, below of=b8]{$\cdots$};
\node(c9)[node distance = 1.5cm, below of=b9]{};
\node(c95)[node distance = 1cm, below of=b95]{};
\node(c10)[node distance = 1.5cm, below of=b10]{$\vdots$};
\node(c11)[node distance = 1.5cm, below of=b11]{$\cdots$};

%Row d
%\node(b5)[node distance = 1cm, below of=a5]{};
\node(d6)[node distance = 3.2cm, below of=b6]{$\D_{k}$};
\node(d7)[node distance = 3.2cm, below of=b7]{};
\node(d8)[node distance = 3.2cm, below of=b8]{$\cdots$};
\node(d9)[node distance = 3.2cm, below of=b9]{};
\node(d95)[node distance = 3.2cm, below of=b95]{$\cdots$};
\node(d10)[node distance = 3.2cm, below of=b10]{$\D_{k}\bigwedge^{\cn-c-k}$};
\node(d11)[node distance = 3.2cm, below of=b11]{$\D_{2}\bigwedge^{\cn-c-k+1}$};
\node(d12)[node distance = 3.2cm, below of=b12]{};

%Row e
\node(e7)[node distance = .9cm, below of=d7]{};
\node(e8)[node distance = .9cm, below of=d8]{$\cdots$};
\node(e9)[node distance = .9cm, below of=d9]{};
\node(e95)[node distance = .9cm, below of=d95]{$\cdots$};
\node(e10)[node distance = .9cm, below of=d10]{$\vdots$};
\node(e11)[node distance = .9cm, below of=d11]{$\cdots$};

%%%%%%%%%
%horizontal arrows
%Row 0
\draw[<-](0)to node{} (1);
\draw[<-](1)to node{} (2);
\draw[<-](2)to node{} (3);
\draw[<-](5)to node{} (6);
\draw[<-](6)to node{} (7);
\draw[<-](95)to node{} (10);
\draw[<-](10)to node{} (11);
\draw[<-](11)to node{} (12);

%Row a
\draw[<-](a1)to node{} (a2);
\draw[<-](a2)to node{} (a3);
\draw[<-](a5)to node{} (a6);
\draw[<-](a6)to node{} (a7);
\draw[<-](a95)to node{} (a10);
\draw[<-](a10)to node{} (a11);
\draw[<-](a11)to node{} (a12);

%Row b
\draw[<-](b2)to node{} (b3);
\draw[<-](b5)to node{} (b6);
\draw[<-](b6)to node{} (b7);
\draw[<-](b95)to node{} (b10);
\draw[<-](b10)to node{} (b11);
\draw[<-](b11)to node{} (b12);

%Row d
\draw[<-](d6)to node{} (d7);
\draw[<-](d95)to node{} (d10);
\draw[<-](d10)to node{} (d11);
\draw[<-](d11)to node{} (d12);
%%
%Upper part:
%%%%%%%%%%%%Above the middle
 %Row a'
\node(a0')[node distance = 1.15cm, above of=0]{};
\node(a1')[node distance = 1.15cm, right of=a0']{$\bigwedge^{c+1}$};
\node(a2')[node distance = 1.2cm, right of=a1']{$\bigwedge^{c+2}$};
\node(a3')[node distance = 1cm, right of=a2']{};
\node(a4')[node distance = 1cm, right of=a3']{$\cdots$};
\node(a5')[node distance = 1cm, right of=a4']{};
\node(a6')[node distance = 1.2cm, right of=a5']{$\bigwedge^{c+k}$};
\node(a7')[node distance = 1.2cm, right of=a6']{};
\node(a8')[node distance = .5cm, right of=a7']{$\cdots$};
\node(a9')[node distance = .5cm, right of=a8']{};
\node(a95')[node distance = 1cm, right of=a9']{$\bigwedge^{\cn-1}$};
\node(a10')[node distance = 1.7cm, right of=a95']{$\bigwedge^{\cn}$};
\node(a11')[node distance = 2.2cm, right of=a10']{0};
\node(a12')[node distance = 1.4cm, right of=a11']{};

%Row b'
\node(b'2)[node distance = 1cm, above of=a2']{$\S_{1}\bigwedge^{c+3}$};
\node(b'3)[node distance = 1cm, above of=a3']{};
\node(b'4)[node distance = 1cm, above of=a4']{$\cdots$};
\node(b'5)[node distance = 1cm, above of=a5']{$$};
\node(b'6)[node distance = 1cm, above of=a6']{$\S_{1}\bigwedge^{c+k+1}$};
\node(b'7)[node distance = 1cm, above of=a7']{};
\node(b'8)[node distance = 1cm, above of=a8']{$\cdots$};
\node(b'9)[node distance = 1cm, above of=a9']{};
\node(b'95)[node distance = 1cm, above of=a95']{$\S_{1}\bigwedge^{\cn}$};
\node(b'10)[node distance = 1cm, above of=a10']{$0$};

%Row c'
\node(c'3)[node distance = 1cm, above of=b'3]{};
\node(c'4)[node distance = 1cm, above of=b'4]{};
\node(c'6)[node distance = 1.3cm, above of=b'6]{$\vdots$};
\node(c'7)[node distance = 1cm, above of=b'7]{};
\node(c'8)[node distance = 1cm, above of=b'8]{};

%Row d'
\node(d'6)[node distance = 3.2cm, above of=b'6]{$\S_{k}\bigwedge^{\cn}$};
\node(d'7)[node distance = 3.2cm, above of=b'8]{0};

%%Arrows:
%Row a'
\draw[<-](a1')to node{} (a2');
\draw[<-](a2')to node{} (a3');
\draw[<-](a5')to node{} (a6');
\draw[<-](a6')to node{} (a7');
\draw[<-](a95')to node{} (a10');
\draw[<-](a10')to node{} (a11');

%Row b'
\draw[<-](b'2)to node{} (b'3);
\draw[<-](b'5)to node{} (b'6);
\draw[<-](b'6)to node{} (b'7);
\draw[<-](b'95)to node{} (b'10);

%Row c'
\draw[<-](d'6)to node{} (d'7);

%vertical arrows
%a to 0
\draw [style=dashed] (a1) to node{} (0);
\draw[->] (a1) to node{} (1);
\draw[->] (a2) to node{} (2);
\draw[->] (a6) to node{} (6);
\draw[->] (a10) to node{} (10);

%b to a
\draw [style=dashed] (b2) to node{} (a1);
\draw[->] (b2) to node{} (a2);
\draw[->] (b6) to node{} (a6);
\draw[->] (b10) to node{} (a10);

%c to b
\draw [style=dashed] (d6) to node{} (b2);

%d to c
\draw [style=dashed] (e7) to node{} (d6); dt

%0 to a'
\draw [style=dashed] (a1') to node{} (0);
\draw[<-] (a1') [style=ultra thick] to node{} (1);
\draw[<-] (a2') [style=ultra thick]to node{} (2);
\draw[<-] (a6') [style=ultra thick]to node{} (6);
\draw[<-] (a10')[style=ultra thick]to node{} (10);

%a' to b'
\draw [style=dashed] (b'2) to node{} (a1');
\draw[<-] (b'2) to node{} (a2');
\draw[<-] (b'6) to node{} (a6');
\draw[<-] (b'95) to node{} (a95');

%b' to d'
\draw [style=dashed] (d'6) to node{} (b'2);
\end{tikzpicture}
\label{Figure3}
{\bf Figure \ref{Figure3}}
\end{center}

\let\thefootnote\relax\footnote{
\noindent AMS Subject Classification:
Primary: 14H99,
Secondary: 13D02, 14H51 \smallbreak
The first author is grateful to the
National Science Foundation for partial support. This work is a contribution to Project I.6 of the second author within the SFB-TRR 195 "Symbolic Tools in Mathematics and their Application" of the German Research Foundation (DFG).}
\bibliographystyle{ABC99}

\bigskip

\end{document}